\documentclass[a4paper,leqno,12pt]{article}

\usepackage[utf8]{inputenc}
\usepackage[T1]{fontenc}
\usepackage[english]{babel}

\usepackage{crimson}

\usepackage{geometry}

\usepackage{fancyhdr}

\usepackage{hyperref}
\hypersetup{
	colorlinks   = true, 
	urlcolor     = blue, 
	linkcolor    = blue, 
	citecolor   = black 
}

\usepackage[inline]{enumitem}

\usepackage{graphicx}

\usepackage{color}

\usepackage[intlimits]{amsmath}
\usepackage{amssymb, amsfonts}

\usepackage[thmmarks, amsmath, amsthm]{ntheorem}

\usepackage{MnSymbol}
\usepackage{mathbbol}

\usepackage{bbm}

\usepackage{mathrsfs}

\usepackage[all]{xy}

\usepackage{tikz-cd}
\usepackage{cases}

\numberwithin{equation}{section}
\newtheorem{theoremcounter}{theoremcounter}[section]

\theoremstyle{plain}

\newtheorem{corollary}[theoremcounter]{Corollary}
\newtheorem{lemma}[theoremcounter]{Lemma}
\newtheorem{proposition}[theoremcounter]{Proposition}

\newtheorem{theorem}[theoremcounter]{Theorem}

\newtheorem{example}[theoremcounter]{Example}

\theoremnumbering{Alph}
\newtheorem{introtheorem}{Theorem}
\newtheorem{introcorollary}[introtheorem]{Corollary}

\theoremstyle{definition}

\newtheorem{definition}[theoremcounter]{Definition}

\theoremstyle{remark}

\newtheorem{remark}[theoremcounter]{Remark}

\usepackage{xargs}                      
\usepackage[colorinlistoftodos,prependcaption,textsize=tiny]{todonotes}
\newcommandx{\unsure}[2][1=]{\todo[linecolor=red,backgroundcolor=red!25,bordercolor=red,#1]{#2}}
\newcommandx{\change}[2][1=]{\todo[linecolor=blue,backgroundcolor=blue!25,bordercolor=blue,#1]{#2}}
\newcommandx{\info}[2][1=]{\todo[linecolor=OliveGreen,backgroundcolor=OliveGreen!25,bordercolor=OliveGreen,#1]{#2}}
\newcommandx{\improvement}[2][1=]{\todo[linecolor=Plum,backgroundcolor=Plum!25,bordercolor=Plum,#1]{#2}}

\usepackage[nameinlink, capitalize, noabbrev]{cleveref}

\newcommand{\Cs}{\mathrm{C}^*}
\newcommand{\Csr}{\mathrm{C}^*_{\text{red}}}

\newcommand{\ra}{\rightarrow}

\newcommand{\bC} {\mathbb C}

\newcommand{\authors}{Sanaz Pooya $\bullet$ Baiying Ren $\bullet$ Hang Wang }
\renewcommand{\title}{Euler characteristics, higher Kazhdan projections and delocalised $\ell^2$-Betti numbers}

\begin{document}
	

	\thispagestyle{empty}
	
	\noindent
	\begin{minipage}{\linewidth}
		\begin{center}
			\textbf{\Large \title} \\
			\authors    
		\end{center}
	\end{minipage}
	
	\renewcommand{\thefootnote}{}
	\footnotetext{last modified on \today}
	\footnotetext{%
	}
	
	\vspace{2em}
	\noindent
	\begin{minipage}{\linewidth}
		\textbf{Abstract}. For non-amenable finitely generated virtually free groups, we show that the combinatorial Euler characteristic introduced by Emerson and Meyer is the preimage of the K-theory class of higher Kazhdan projections under the Baum-Connes assembly map. This allows to represent the K-theory class of their higher Kazhdan projection as a finite alternating sum of the K-theory classes of certain averaging projections. The latter is associated to the finite subgroups appearing in the fundamental domain of their Bass-Serre tree. As an immediate application we obtain 
        non-vanishing calculations for delocalised $\ell ^2$-Betti numbers for this class of groups.
		
	\end{minipage}

	
\section{Introduction}
	\label{sec:introduction}

Topological invariants play a fundamental role in understanding the geometry and topology of spaces, groups, and spaces equipped with symmetries. Typical examples include the Euler characteristic for finite CW complexes and its refinement, Betti numbers, which capture the ranks of homology groups. This classical setting can be extended by considering universal covering spaces acted upon by fundamental groups through deck transformations. The classic $L^2$-invariants and a younger analogue concept of delocalised $L^2$-invariants~\cite{lott1999} have emerged as powerful tools in geometry and topology. A notable application of the latter can be found in \cite{XieYu2021}, where it is shown that certain delocalised $L^2$-invariant can obstruct the surjectivity of the Baum–Connes assembly map through their non-algebraicity.
The analogues of Betti numbers in this setting are the $\ell^2$-Betti numbers and delocalised $\ell^2$-Betti numbers. These notions can be promoted to invariants for discrete groups by replacing the space with the universal cover of the classifying space of the group. These invariants have deep connections to representation theory, geometric group theory, and noncommutative geometry.

The key objects of study in this paper are higher Kazhdan projections associated with a discrete group, introduced  in ~\cite{linowakpooya2020}. These projections serve as group invariants, appearing as elements in the $K$-theory of group $C^*$-algebras, associated to a family of unitary representations, and providing a noncommutative counterpart to classical topological invariants. In particular, for a group $G$ with its left regular representation, the higher Kazhdan projections can define non-zero elements in the $K$-theory of the reduced group $C^*$-algebra:
\[
 [p_n] \in \mathrm K_0(\Csr(G)), \qquad n=1,2, \ldots.
\]
These $K$-theory classes are closely linked to $\ell^2$-Betti numbers via the von Neumann trace  and delocalised $\ell^2$-Betti numbers via delocalised traces~\cite{Pooya-Wang}. As a result, they are recognised as higher topological invariants for groups. However, despite their significance, explicit computations of these projections remain highly non-trivial. In ~\cite{Pooya-Wang}, the first and third author computed higher Kazhdan projections for $\mathbb{Z}_m * \mathbb{Z}_n$ directly by solving equations in the group ring.
The result was later naturally generalized to amalgamted product groups $\mathbb{Z}_m *_{\mathbb{Z}_d} \mathbb{Z}_n$ in \cite{Ren}.
However, this direct computation method depends heavily on the associated Laplacians and free resolutions, making generalizations to broader classes of groups quite challenging.

The central goal of this paper is to establish a relationship between higher Kazhdan projections and the $K$-homological Euler characteristic via the Baum-Connes assembly map, particularly for virtually free groups. Our approach provides a topological framework for studying these invariants and highlights their interplay with higher index theory. Guided by this perspective, we propose a new and flexible method for computing higher Kazhdan projections, leveraging the fact that the $K$-group accommodating these projections is itself a group invariant. The computation of this $K$-group is facilitated by the Baum-Connes conjecture, which posits that the Baum-Connes assembly map
\[
\mu_i^G: \mathrm K^G_i(\underline{E}G) \to \mathrm K_i(\Csr(G)) \qquad i=0,1
\]
is an isomorphism (see~\cite{baumconneshigson1994}). 

Our main result relates the K-theory classes of higher Kazhdan projections, provided they lie in the image of the Baum-Connes assembly map, to the equivariant Euler characteristic on the universal space of proper actions. See Theorem~\ref{thm A} below.

The Euler characteristic, a fundamental topological invariant, serves as a measure of the shape and structure of a space. By the Gauss-Bonnet-Chern theorem, it appears as the Fredholm index of the de Rham operator. For closed manifolds, such an operator defines an element in the $K$-homology of the manifold corresponding to the Euler number through the $K$-homological functor applied to the inclusion of a point~\cite{Rosenberg99}. Hence, the K-homology class of the de Rham operator is called the geometric Euler characteristic. L\"uck and Rosenberg~\cite{Lueck-Rosenberg} extended this notion to an arbitrary manifold carrying a proper cocompact action by a discrete group, and Emerson and Meyer~\cite{EM} later provided a combinatorial description for more general proper cocompact complexes. Since the image of the combinatorial Euler characteristic under the assembly map can be computed explicitly, we obtain explicit computations of higher Kazhdan projections and delocalised $\ell^2$-Betti numbers. See Corollary~\ref{Corrollary B} and Corollary~\ref{Corrollary C} below. 
Let $SX$ denote the set of simplices of a complex $X$. Let  $\rho_{\sigma}$ denote the averaging projection associated to the (finite) isotropy group $G_{\sigma}$ of the barycentre $\xi_{\sigma}$ of a simplex $\sigma$ in $SX$.
\begin{introtheorem}(See Theorem \ref{main thm1})\label{thm A} 
    Let $G$ be a discrete group for which there is a $G$-finite model $X$ for $\underline E G$. 
    Assume that there is only one non-vanishing higher Kazhdan projection $p_n$ associated with $G$, for some $n$.  
    Then there is an explicit formula for the $\mathrm{K}$-theory class $[p_n]$ employing the averaging projections associated with the orbits
    \begin{equation*}
    (-1)^n[p_n] = \sum_{\sigma\in G \backslash SX}(-1)^{|\sigma|} [\rho_{\sigma}] \in \mathrm{K_0(C^*_{red}}(G)).
    \end{equation*}
     Moreover, the combinatorial Euler class   
   $
    [\mathrm{Eul}^{\mathrm{cmb}}_X] \in 
    \mathrm{K}^G_0(\underline{E}G)
    $
    satisfies
    \begin{equation*}
        \mu_0^G([\mathrm{Eul}^{\mathrm{cmb}}_X]) = (-1)^n[p_n].
    \end{equation*}
\end{introtheorem}
The proof of our main result relies on the flexibility of higher Kazhdan projections as topological invariants. A key insight is the analogy with Hodge theory, where these projections naturally arise as kernels of Laplacians, as well as their interpretation within the framework of reduced group cohomology. Unlike direct analytic computations of Laplacian kernels, which are often highly intricate, our method circumvents the need for solving explicit equations, offering a significant conceptual advantage. This approach significantly generalises our previous results for $PSL(2, \mathbb{Z})$ in~\cite{Pooya-Wang} and provides a systematic way to to calculate delocalised $\ell^2$-Betti numbers. Furthermore, to establish our main theorem, we introduce higher Kazhdan projections in a broader setting, further expanding their applicability in noncommutative geometry and geometric group theory. (See Definition \ref{def: higher kazhdan gen})
 
For the class of virtually free groups Theorem \ref{thm A} has the following applications.

 \begin{introcorollary}(See Corollary \ref{cor: tree})\label{Corrollary B}
    Let $G$ be a non-amenable finitely generated virtually free group and
    $X$ be its Bass-Serre tree. Then the $\mathrm{K}$-theory class $[p_1]$ can be described in terms of the averaging projections associated to the vertices and edges of the fundamental domain of the tree $X$
    \begin{equation*}
        -[p_1]=\sum_{v\in \, \mathrm{vert}\ G \backslash X}{[\rho_{v}]}\,\,\,\,\,-\sum_{e\in \, \mathrm{edge}\  G \backslash X}{[\rho_{e}]}\in \mathrm{K_0(C^*_{red}}(G)).
    \end{equation*}    
\end{introcorollary}
\begin{remark}
    Amenable groups $G$ are excluded in the above corollary because, in this case, the group cohomology $\mathrm H^1(G, \ell^2(G))$ is not Hausdorff hence not reduced, which prevents the existence of higher Kazhdan projections in degree 1. (See \cite[Corollary III.2.4]{Guichardet} and \cite[Proposition 16]{Bader-Nowak}.)
\end{remark}
We denote by $\langle{g}\rangle$ the conjugacy class of $g\in G$. In \cite{Pooya-Wang} we defined the notion of delocalised $\ell^2$-Betti numbers for hyperbolic groups as K-theory pairing between the delocalised traces associated with the conjugacy classes and the K-theory class of higher Kazhdan projections. (See Section \ref{sec: Examples})

Let $\mathcal {F}_G :=  \left\langle \left\{ \frac{1}{|F|} \;\middle|\; F \leq G \text{ finite} \right\}\right\rangle \subseteq \mathbb{Q} $ denote the additive subgroup of $\mathbb Q$ generated by the inverses of the orders of finite subgroups of $G$.

\begin{introcorollary} (See Corollary \ref{cor: nonzero del betti})\label{Corrollary C}
    Let $G$ be a non-amenable finitely generated virtually free group acting properly on a tree $X$, and let 
    $g\in G$ fix vertices, but no edges of $X$.
    Then the first delocalised $\ell^2$-Betti numbers of $G$ are non-zero and they satisfy
    \[
       \beta^{(2)}_{1, \langle{g}\rangle} (G) \in \mathcal {F}_G.  
       \]    
\end{introcorollary}

The paper is organised as follows: Section~\ref{sec: Kazhdan projections} extends the definition of Kazhdan projections to the broader context of type $FP_{n}$ groups, beyond those of type $F_n$. Section~\ref{sec: Euler characteristics} reviews the notion of $K$-homological Euler characteristics and their image under the Baum-Connes assembly map. Section~\ref{sec: Main theorem} presents our main result, establishing a connection between higher Kazhdan projections as group invariants with equivariant Euler characteristics as space invariants, via the Baum-Connes assembly map. In Section \ref{sec: Examples} we recall delocalised $\ell^2$-Betti numbers and describe it for groups satisfying Theorem \ref{main thm1}. 
Finally, Section~\ref{sec: Applications} provides applications to virtually free groups and gives an explicit expression for the K-theory class of higher Kazhdan projections as a sum of averaging projections associated to the vertices and edges of the Bass–Serre tree. This yields concrete non-vanishing computations for delocalised $\ell^2$-Betti numbers.

\section*{Acknowledgment} 
HW is supported
by the grants 23JC1401900, NSFC 12271165 and in part by Science and Technology Commission of
Shanghai Municipality (No. 22DZ2229014).
The authors would like to thank 
Alain Valette for his helpful comments that help to improve the presentation.

\section{Higher Kazhdan projections revisited}
\label{sec: Kazhdan projections} 
In this section we provide a more general setting for defining higher Kazhdan projections. The standard assumption in their definition (cf. \cite [Definition 4]{linowakpooya2020}) is that the group $G$ is of type $F_n$. We now generalise this to groups of type $FP_n$, which allows for applications beyond the topological setting.

Let us briefly recall the relevant finiteness properties. 
A group $G$ is of type $F_n$ if the classifying space $BG$ admits a model whose $n$-skeleton is finite.
A group $G$ is of type $FP_{n}$ over $\mathbb C$ if the trivial $\mathbb CG$-module $\mathbb C$ admits a resolution
\begin{equation}\label{fp res}
  \cdots \ra N_{k} \ra \cdots \ra N_{n} \stackrel{\delta_n}{\ra} N_{n-1} \ra \cdots \ra N_1\stackrel{\delta_0}{\ra} N_0{\ra} \mathbb C \ra 0,
\end{equation}
where each $N_i$ for $i\leq n$ is a finitely generated projective $\mathbb CG$-module. 

Let $(\pi, H)$ be a unitary representation of $G$. In \cite{linowakpooya2020} higher Kazhdan projections were defined under the assumption that $G$ is type $F_{n+1}$ over $\mathbb R$.  As we will soon switch to $\mathbb CG$-modules rather than $\mathbb RG$-modules, we discuss the process of defining higher Kazhdan projections directly for free $\mathbb CG$-modules.

Given a model of $BG$, consider a free resolution of the trivial $\mathbb CG$- module $\mathbb C$ 
\[
	\cdots \ra \mathbb C G^{\oplus k_2} \rightarrow \mathbb C G^{\oplus k_1} \rightarrow \mathbb C G \ra  \mathbb C \ra 0,
\]
where $k_i$ denotes the number of $i$-cells in the chosen model of $BG$, and the complex is equipped with the standard augmentation. 
Applying $\mathrm{Hom}_{\mathbb C G}(\cdot, H)$, and using the identification 
$ \mathrm{Hom}_{\mathbb C G} (BG^{(n)}, H) \simeq H^{\oplus k_n}$, 
we obtain a cochain complex of Hilbert spaces and bounded maps $d^i$ 
	\[
           H 
          \xrightarrow{d^0} 
        H^{\oplus k_1} 
		\rightarrow
		\cdots 
		\rightarrow H^{\oplus k_{n-1}} \xrightarrow{d^{n-1}} 
		H^{\oplus k_n} \xrightarrow{d^{n}} 
		H^{\oplus k_{n+1}} 
		\rightarrow 
		\cdots.
	\]

Using this complex we define the higher Laplace operator $\Delta_n$ acting on $H^{\oplus k_n}$ by
\[
  \Delta_n = d^{n} (d^{n})^* + (d^{n+1})^* d^{n+1} \in \mathrm M _{k_n} (\mathbb CG)
\]
More generally, given a unitary representation $(\pi,H)$ of $G$ consider the operator $\pi(\Delta_n) \in \pi \overline{(\mathbb CG)}$ (we will skip the reference to the $\pi$ if it will be clear from the context). 

Higher Laplace operators are related to the notion of reduced cohomology. The $n$-th reduced cohomology $\overline{\mathrm H}^n$ is defined as the quotient 
$\ker d^n / {\overline {\mathrm{im} d^{n-1}}}$, where ${\overline {\mathrm{im} d^{n-1}}}$ is the closure of the $n-1$-coboundary.  Note that $\overline{\mathrm H}^n$ is a subspace of $\ker d^n$ and hence a Hilbert space. Cohomology in degree $n$ is reduced if $\overline{\mathrm H}^n = \mathrm H^n$.
In this case, by the Hodge-de Rahm isomorphism we have
\[
  \ker(\Delta_n) \simeq \overline{\mathrm H}^n(G, H)
\]

As in \cite{linowakpooya2020} a higher Kazhdan projection in degree $n$ associated to a chosen model of $BG$ and a given family $\mathcal F$ of unitary representations is the projection $p_n \in M_{k_n}(\Cs_{\mathcal F} (G))$ whose image under any unitary representation in $\mathcal F$ is the orthogonal projection 
$$p_n\colon H^ {\oplus k_n} \ra \ker(\Delta_n).$$ 

We now generalise this setup. 
Assume that $G$ is of type $FP_{n+1}$ over $\mathbb C$, and consider a resolution as in (\ref{fp res}).  Since $N_i$'s are finitely generated projective $\mathbb CG$-modules, for $i\leq n+1$, we have $N_i \simeq \mathbb CG^{\oplus k_i}\, q_i$, where $q_i \in M_{k_i}(\mathbb CG)$. Each connecting map $d^i \colon N_{i} \ra N_{i+1}$ is defined by $d^i(vq_{i}) = (T_iv)q_{i+1}$, where the matrix $T_i \in M_{k_i\times k_{i+1}} (\mathbb CG)$ and $q_i$ and $q_{i+1}$ are the projections associated with $N_i$ and $N_{i+1}$, respectively. Given $(\pi, H)$, applying
the functor $\mathrm {Hom}_{\mathbb C G}(\cdot, H)$ gives rise to the cochain complex $$C^i(N_{\bullet}, H):=\mathrm{Hom_{\mathbb C G}}(N_{\bullet}, H),$$
where each $C^i(N_i, H)=\mathrm{Hom_{\mathbb C G}}(N_i, H) \simeq H^{\oplus k_i}q_i$ is a Hilbert space. By abuse of notation we still denote by $d^i$ the coboundaries in the cochain complex of Hilbert spaces.
The higher Laplace operator $\Delta_n$ can be defined as before
$$\pi(\Delta_n) = d^{n} (d^{n})^* + (d^{n+1})^* d^{n+1} \in q_n M_{k_n}(\pi(\mathbb CG)) q_n \subseteq  M_{k_n}(\pi(\mathbb CG))$$ acting on $C^n(N_n, H)$. Thanks to the Hodge-de Rham isomorphism its Kernel is the $n$-reduced cohomology group.

This allows us to formulate the following definition.

\begin{definition} \label{def: higher kazhdan gen}
    Let $G$ be a group of type $FP_{n+1}$ over $\mathbb C$. Let $\mathcal F $ be a family of unitary representations of $G$. A higher Kazhdan projection $p_n$ in degree $n$ is a projection in (a corner of) 
    $\mathrm{M}_{k_n}(\Cs_ {\mathcal F}(G))$ for some $k_n \in \mathbb N$ whose image under any unitary representation $(\pi, H) \in \mathcal F$ is the orthogonal projection 
    \[p_n \colon C^n(N_n, H) \ra \ker \pi(\Delta_n).
    \]
\end{definition}

We can view higher Kazhdan projection $p_n$ as the strong operator  limit of a family of heat operators
$p_n= \lim_{t\to \infty} e^{-t\pi(\Delta_n)}$ for each $(\pi, H)$. This operator converges in norm to 
$p_n \in M_{k_n}(\Cs_{\mathcal F}(G))$ provided $\pi(\Delta_n)$ has a spectral gap.
A characterisation for existence of spectral gap for the higher Laplace operator $\Delta_n$ in terms of reduced cohomology is given in \cite{Bader-Nowak}.  This will be our tool to verify existence of spectral gap.
\begin{proposition} \label{lem: spectral gap tool}
\cite[Proposition 16]{Bader-Nowak}
    Let $(\pi, H)$ be a unitary representation of $G$. The higher Laplace operator $\pi(\Delta_n)$ has a spectral gap in $\mathrm{M}_{k_n}(\Cs_{\pi} {\mathcal n}(G))$ if and only if $\mathrm H^n(G, H)$ and $\mathrm H^{n+1}(G, H)$ are reduced. 
\end{proposition}

In this article, we will be working the left regular representation $(\lambda, \ell^2(G))$. Under the assumption of existence of a spectral gap for the associated higher Laplace operator $\Delta_n$, the projection $p_n$ defines a class $[p_n] \in \mathrm K_0(\Csr(G))$.
\section{Combinatorial Euler characteristics}
\label{sec: Euler characteristics}

In this section, we recall the construction of the combinatorial Euler characteristics given by Emerson and Meyer \cite{EM}.

Let $G$ be a discrete group. Let $X$ be a  countable, locally finite, simplicial complex on which $G$ acts continuously and simplicially.
Denote by $SX$ the set of all simplices of $X$, with discrete topology and the induced action of $G$.
Denote by $\xi:SX\to X$ the map that identifies each simplex with its barycenter.
Then $\xi$ induces the following $G$-equivariant $*$-morphism:
$$
\xi^*:C_0(X) \to C_0(SX).
$$

Equip the Hilbert space $l^2(SX)$ with the $\mathbb Z_2$-grading given by the parity of simplicies, i.e., 
for $\sigma\in SX$, the indicator function $\delta_{\sigma}$ belongs to $l^2(S^+X)$ (resp. $l^2(S^-X)$) if $\sigma$ is an even (resp. odd) dimensional simplex.
The action of the algebra $C_0(SX)$ on $l^2(S^{\pm}X)$ by point-wise multiplication induces a $G$-equivariant $*$-morphism, denoted by $\phi$: 
\[
\phi: C_0(SX)\rightarrow \mathcal{B}(l^2(SX)). 
\]
By the discreteness of the space $SX$, the image of $\phi$ lies in the algebra $\mathcal{K}(l^2(SX))$ of compact operators. 
Moreover, the composition of $\phi$ with ${\xi}^*$ induces a $G$-equivariant $*$-morphism:
\[
\phi\circ {\xi}^*: C_0(X)\rightarrow \mathcal{K}(l^2(SX)), 
\]
which yields a $G$-equivariant $KK$-theory class
\[
[(l^2(SX), \phi\circ {\xi}^*, 0)]\in \mathrm{KK}^G_0(C_0(X), \mathbb C). 
\]

\begin{definition} \cite[Definition 22]{EM}
    Let $G$ be a discrete group. Let $X$ be a simplicial complex as above. The combinatorial $G$-equivariant Euler characteristic of $X$ is given by
    \[
    [\mathrm{Eul}_{X}^{cmb}]:=[(l^2(SX), \phi\circ {\xi}^*, 0)]\in \mathrm{KK}^G_0(C_0(X), \mathbb C).
    \]
\end{definition}

\textbf{Notation}: Throughout this paper, for a finite subgroup $H$ of $G$, we denote by $\rho_H$ the averaging projection in the group algebra $\mathbb{C}H$ associated to the trivial representation:
$$
\rho_H:=\frac{1}{|H|}\sum_{h\in H} h,
$$
where $|H|$ stands for the cardinality of $H$. Moreover, when $X$ is a model for $\underline EG$, we denote by $\rho_{\sigma}$ the averaging projection associated to the finite isotropy group $G_{\sigma}$ of the barycentre $\xi_{\sigma}$ of a simplex $\sigma$ in $SX$. These projections belong to $\mathbb{C}G$ by inclusion.
\begin{remark}
    Without loss of generality, we require that $G_{\sigma}$ is the same as the isotropy group of a simplex $\sigma$ in $SX$.
\end{remark}



     We recall the description of combinatorial Euler characteristics in Section 5 of \cite{EM} as follows. 
     Let $X$ be a $G$-simplicial complex.
     For a subgroup $H\subseteq G$ and each connected component $A$ of the set $X^H$ of fixed points under the action of $H$, let $\mathrm{dim}_{H,A}\in \mathrm{KK}^{G}_0(C_0(X),\bC)$ be the class given by the homomorphism:
     \begin{equation*}
              C_0(X) \to C_0(G/H) \subseteq
     \mathcal{K}(l^2(G/H))
     \end{equation*}
sending $f\in C_0(X)$ to the operator of multiplication by the  function $gH \mapsto f(gx)$,        
     where $x$ is a point in $A$. It is clear that $\mathrm{dim}_{H,A}$ is independent of the choice of $x$ by homotopy invariance of Kasparov cycles.
     Moreover, it can be checked that $\mathrm{dim}_{gHg^{-1},gA}=\mathrm{dim}_{H,A}$ for every $g\in G$ due to unitary equivalence, and especially, $\mathrm{dim}_{H,gA}=\mathrm{dim}_{H,A}$ if $g$ belongs to the normaliser $N(H)$ of $H$.
     Let $S(H,A)\subseteq SX$ be the set of all simplices of $A$ whose isotropy group is exactly $H$ and let $\chi(X,H,A)$ denote the alternating sum of the numbers of $n$-simplices in $N(H)\backslash S(H,A)$. 
     Let $C(G)$ denote the set consisting of conjugacy classes $\langle H\rangle$ of subgroups $H$ such that $H$ is the isotropy group of some simplex in $X$. For a conjugacy class $\langle H\rangle \in C(G) $, let $A$ run through the set of connected components of $N(H)\backslash X^H$.
     Then the combinatorial Euler characteristic of $X$ can be written as
    \begin{equation*}
    \label{sum}
    [\mathrm{Eul}^{\mathrm{cmb}}_X]=\sum_{\langle H\rangle\in C(G),A} \chi(X,H,A)\cdot\mathrm{dim}_{H,A}.
    \end{equation*}

    Assume that $X$ is a model for $\underline EG$, then the isotropy groups are finite, and so are the elements in $ C(G)$.
    Moreover, due to Theorem 1.9 in \cite{Lueck05}, $X^H$ is connected for each finite subgroup $H$.
    Hence, the reference to  $A$ in the above equality can be removed, and we will write $\mathrm{dim}_{H,A}$, $S(H,A)$ and $\chi(X,H,A)$ for $\mathrm{dim}_{H}$, $S(H)$ and $\chi(X,H)$, respectively.
    Thus, the equality above can be simplified to 
    \begin{equation} \label{dim}
        [\mathrm{Eul}^{\mathrm{cmb}}_X]=\sum_{\langle H\rangle\in C(G)} \chi(X,H)\cdot\mathrm{dim}_{H}.
    \end{equation}
   
The following statement was mentioned in 
Section~5 of \cite{EM}. We provide a proof here for the readers' convenience.  
\begin{lemma}
\label{lemma: Euler}
Let $G$ be a discrete group. Suppose there exists a cocompact model $X$ for $\underline EG$. Then the image of the combinatorial Euler characteristic under the Baum–Connes assembly map $\mu_G$ is given by 
 $$
  \mu_0^G([\mathrm{Eul}^{\mathrm{cmb}}_X])=\sum_{\sigma\in G \backslash SX}{{(-1)}^{|\sigma|}[\rho_{\sigma}]}\in \mathrm{K}_0(\Csr(G)),
 $$
 where $SX$ denotes the set of simplices of $X$,  $|\sigma|$ is the dimension of the simplex $\sigma$, and $\rho_{\sigma}$ is the projection associated with $\sigma$.
\end{lemma}

\begin{proof}
 We need to show that
    \begin{equation}
        \label{sum of dimH}
        \sum_{\langle H\rangle\in C(G)} \chi(X,H)\cdot\mathrm{dim}_{H}=\sum_{\sigma\in G \backslash SX}{(-1)}^{|\sigma|}\cdot\mathrm{dim}_{G_{\sigma}}.
    \end{equation}
    This together with (\ref{dim}) then yields the result.
     Consider the following surjective map:
    \begin{align*}
       \varphi: \{ \text{orbits in}\ G\backslash SX \} &\longrightarrow C(G)\\
        G\sigma &\mapsto \langle G_{\sigma}\rangle,\quad \forall \sigma\in SX.
    \end{align*}
    Fix an arbitrary element $\langle H \rangle\in C(G)$. Its preimage is
    $\varphi^{-1}(\langle H \rangle)$ is given by the union of certain finitely many disjoint orbits.
    By choosing a representative for each orbit such that its isotropy group is exactly $H$, we obtain
    \begin{align}
    \label{orbit}
        \varphi^{-1}(\langle H \rangle) = \bigsqcup_{\sigma\in r\left(\varphi^{-1}(\langle H\rangle\right))}G\sigma,
    \end{align}
    for a finite subset $r\left(\varphi^{-1}(\langle H\rangle)\right)$ of $SX$ satisfying $\forall \sigma \in r\left(\varphi^{-1}(\langle H\rangle\right))$, $G_{\sigma}=H$. 
    Since $S(H)$ is the set of all simplices (of $X^H$) whose isotropy group is exactly $H$ (\ref{orbit}) implies that
    $$
    S(H)=\bigsqcup_{\sigma\in r\left(\varphi^{-1}(\langle H\rangle\right))}N(H)\sigma.
    $$
    Hence,
    $$
    N(H)\backslash S(H)=\left\{[\sigma]|\sigma\in r\left(\varphi^{-1}(\langle H\rangle\right) \right\},
    $$
    where $[\sigma]$ denotes the class of $\sigma$ in $N(H)\backslash S(H)$. 
    Since $\chi(X,H)$ is by definition the alternating sum of numbers of $n$-simplices in $N(H)\backslash S(H)$, we obtain
    $$
    \chi(X,H)= \sum_{\sigma\in r\left(\varphi^{-1}(\langle H\rangle\right))}(-1)^{|\sigma|}.
    $$
    This implies that
    \begin{align} \label{dimH}
        \chi(X,H)\cdot\mathrm{dim}_{H}& = \sum_{\sigma\in r\left(\varphi^{-1}(\langle H\rangle\right))}(-1)^{|\sigma|}\cdot\mathrm{dim}_{H} \notag\\
        &:=\sum_{\sigma\in r\left(\varphi^{-1}(\langle H\rangle\right))}(-1)^{|\sigma|}\cdot\mathrm{dim}_{G_{\sigma}} .
    \end{align}
    For two elements $\langle H \rangle\neq \langle H' \rangle\in C(G)$, it is clear that the orbits of their preimages are disjoint, i.e.,
    $$
    \varphi^{-1}(\langle H\rangle) \cap \varphi^{-1}(\langle H' \rangle)=\bigsqcup_{\sigma\in r\left(\varphi^{-1}(\langle H\rangle\right))}G\sigma \, \bigcap \bigsqcup_{\sigma '\in r\left(\varphi^{-1}(\langle H'\rangle\right))}G\sigma '= \emptyset.
    $$
    Thus, the disjoint union of all preimages of $C(G)$ is a subset of $G\backslash SX$.
    Moreover every orbit $G\sigma \in G\backslash SX$ maps to a conjugacy class $\langle G_{\sigma} \rangle\in C(G)$ under the map $\varphi$, thus it belongs to the preimage  $\varphi^{-1}(\langle G_{\sigma}) \rangle$.
    Therefore, the disjoint union of all preimages of $C(G)$ is exactly $G\backslash SX$. This together with equality (\ref{dimH}) yield
    \begin{equation*}
        \sum_{\langle H\rangle\in C(G)} \chi(X,H)\cdot\mathrm{dim}_{H}=\sum_{\langle H\rangle\in C(G)}\sum_{\sigma\in r\left(\varphi^{-1}(\langle H\rangle\right))}(-1)^{|\sigma|}\cdot\mathrm{dim}_{G_{\sigma}}=\sum_{\sigma\in G\backslash SX}(-1)^{|\sigma|}\cdot\mathrm{dim}_{G_{\sigma}}.
    \end{equation*}
    
    
    Next, we show that for every finite subgroup $H$ that occurs as the isotropy group of some simplex in $X$
    \begin{align}
    \label{mudimH}
    \mu_0^G(\mathrm{dim}_{H})=[\rho_H].
    \end{align}
    The cycle $\mathrm{dim}_{H}=\left[C_0(X)\rightarrow C_0(G/H)\subseteq\mathcal{K}(l^2(G/H))\right]$ in $\mathrm{K}_0^G(X)$
    is the image of 
\[
\left[C_0(G/H)\rightarrow \mathcal{K}(l^2(G/H))\right]\in \mathrm{K}_0^G(G/H)
\]
   under the restriction map $C_0(X)\rightarrow C_0(G/H).$ 
   Since $C_0(G/H)$ is the induced representation of the trivial representation of $H$ to $G$, the cycle
$\left[C_0(G/H)\rightarrow \mathcal{K}(l^2(G/H))\right]$ in $\mathrm{K}_0^G(G/H)$
is the image of the trivial representation 
$[1_H]\in R(H)$ under the isomorphism
\[
R(H)\cong \mathrm{K}_0^G(G/H).
\]
By the naturality of the assembly maps for $H$ and for $G$ under the inclusion $H\subseteq G$, i.e., the commutativity of the following diagram
\begin{equation*}
\xymatrix@+2pc{ R(H) \ar[r]^{\mu_H}_{\cong} \ar[d]_{\cong} & \mathrm{K}_0(\Csr(H)) \ar[d]_{\subset} \\
                 \mathrm{K}_0^G(G/H) \ar[r]^{\mu_G} &      \mathrm{K}_0(C_r^*(G)) }
\end{equation*}
and the fact that $H\subseteq G$ induces the inclusion $\mathrm{K}_0(\Csr(H))\rightarrow \mathrm{K}_0(\Csr(G))$, we have
\[
\mu_G(\mathrm{dim}_{H})=\mu_G(\left[C_0(G/H)\rightarrow \mathcal{K}(l^2(G/H))])=\mu_H([1_H])=[\rho_H\right].
\]
Following our notation we have $\rho_H=\rho_{\sigma}$ when 
$H$ is the isotropy group of a simplex $\sigma$ in $X$.
Therefore, the conclusion follows from (\ref{sum of dimH}) and (\ref{mudimH}).
\end{proof}

\section{Combinatorial Euler characteristics as the preimage}
\label{sec: Main theorem}

This section presents our main result, Theorem \ref{main thm1}, which gives a concrete description of the K-theory class of the higher Kazhdan projection . The key step in our approach is identification of specific alternating sums of averaging projections associated to the image of the combinatorial Euler characteristic. To establish our proof, we first develop several intermediate results.

We assume that \( G \) is a discrete group that admits a \( G \)-compact universal proper \( G \)-space \( X \). Furthermore, we assume that \( X \)
can be chosen to be a $G$-finite simplicial complex equipped with a simplicial $G$-action. 
Here, \( G \)-finiteness means that the quotient space \( G \backslash X \) is finite, which in particular implies that \( X \) is finite dimensional. In short, we require that there is a $G$-finite model $X$ for $\underline E G$.

Let \( SX^i \) denote the set of \( i \)-simplices of \( X \), and let \( \mathbb{C}[SX^i] \) be the complex vector space whose basis is identified with the set $SX^{i}$.

  

We now state our main theorem. Its proof requires several technical results, which we present first. The proof will be given at the end of this section.

\begin{theorem}
\label{main thm1}
    Let $G$ be a discrete group for which there is a $G$-finite model $X$ for $\underline E G$. 
    Assume that there is only one non-vanishing higher Kazhdan $p_n$ projection associated with $G$, for some $n$.  
    Then there is an explicit formula for the $\mathrm{K}$-theory class $[p_n]$ employing the averaging projections associated with the orbits
    \begin{equation}
        \label{p1}
    (-1)^n[p_n] = \sum_{\sigma\in G \backslash SX}(-1)^{|\sigma|} [\rho_{\sigma}] \in \mathrm{K_0}(\Csr(G))).
    \end{equation}
     Moreover, the combinatorial Euler class   
   $
    [\mathrm{Eul}^{\mathrm{cmb}}_X] \in 
    \mathrm{K}^G_0(\underline{E}G)
    $
    satisfies
    \begin{equation}
        \label{p2}
        \mu_0^G([\mathrm{Eul}^{\mathrm{cmb}}_X]) = (-1)^n[p_n].
    \end{equation}
\end{theorem}
Note that, in principle, to define the higher Kazhdan projection $p_n$, the group must be of type $FP_{n+1}$. However, since in the statement above we assume that $\underline E G$ is finite-dimensional, this finiteness condition on the group is no longer necessary. Next proposition demonstrates how to obtain a suitable projective resolution associated with a finite dimensional model. 
\begin{proposition}
    \label{proj resolution}
    Let $G$ be a discrete group with a $G$-finite model $X$ of $\underline{E} G$ of dimension $n+1$. Then the augmented chain complex associated with $X$ yields a finite length projective resolution of the trivial $\bC G$-module $\bC$ 
    \begin{align*} 
		0 \longrightarrow \bC[SX^{n+1}] \stackrel{\partial_{n}}{\longrightarrow} \cdots \stackrel{\partial_1}{\longrightarrow} \bC[SX^1] \stackrel{\partial_0}{\longrightarrow} \bC[SX^0] \stackrel{\epsilon}{\longrightarrow} \bC \to 0,
        \label{chain complex}\tag{*}
	\end{align*}
   where $\partial_{0},\cdots,\partial_{n}$ denote the boundary map and $\epsilon$ denotes the augmentation map. Moreover, for
   $i\in \{0,\cdots, n+1\}$ we have
   $$
   \bC[SX^i] \cong \bigoplus_{k\in \lambda_i}{\bC G\cdot \rho_{H_{i_k}}},
   $$
   for some finite index set $\lambda_i$ and (finitely many) finite subgroups $H_{i_k}$ of $G$. Further, the following two alternating sums coincide:
    \begin{equation}
    \label{alternating sum} 
        \sum_{\sigma\in G \backslash SX}{{(-1)}^{|\sigma|}[\rho_{\sigma}]}=
        \sum_{i=0}^{n+1}\sum_{k\in \lambda_i}(-1)^i[\rho_{H_{i_k}}].
    \end{equation}   
\end{proposition}

\begin{proof}
        Since $X$ is a universal proper $G$-space, and hence  contractible, we have that 
         $$\mathrm H_i(X)=\mathrm H_i(\mathrm{pt})=0,\quad i>0,$$
         and
         $$\mathrm H_0(X)=\mathrm H_0(\mathrm{pt})=\bC,$$
         i.e. its cohomology is concentrated in degree 0, where it is one-dimensional. This implies that the chain complex \eqref{chain complex} is exact, hence a resolution.
        Next, since $X$ is $G$-compact, for each $i$, the set of $i$-simplices $SX^i$ is a finite disjoint union of orbits, i.e.,
    $$
    SX^i= \bigsqcup_{k\in \lambda_i} G.x_{i_k},
    $$
    for some finite index set $\lambda_i$ and $i$-simplices $x_{i_k}$.
    Thus, $\{x_{i_k} \mid k\in \lambda_i \}$ provides a list of representatives of $G\backslash SX^i$.
    The orbit of any $i$-simplex $x$ can be identified with the quotient group of $G$ over $G_x$, the isotropy group of $x$, we obtain $\bC[SX^i]$ is a finite direct sum of $\bC G$-modules:
    $$
    \bC[SX^i]= \bigoplus_{k\in \lambda_i} \bC[G/G_{x_{i_k}}].
    $$
    By properness of the action of $G$ each isotropy group $G_{x_{i_k}}$ is a finite subgroup of $G$. Therefore, 
    $$
    \bC[SX^i]= \bigoplus_{k\in \lambda_i} \bC[G/G_{x_{i_k}}]=\bigoplus_{k\in \lambda_i} \bC G\cdot \rho_{G_{x_{i_k}}},
    $$
    i.e. $\bC[SX^i]$ is a finitely generated projective $\bC G$-module.
    Let $H_{i_k}$ in the statement be the isotropy group $G_{x_{i_k}}$ for every $i_k\in \lambda_i$, $i\in \{0,\cdots, n+1\}$. Equality (\ref{alternating sum}) follows from the fact that the set $\{x_{i_k} \mid k \in \lambda_i, i = 0, \ldots, n+1\}$ forms a complete set of representatives for the $G$-orbits in $SX$, that is, a choice of representatives of $G \backslash SX$.
\end{proof}

\begin{lemma}
\label{step1}
    Let $G$ be a discrete group. Let $X$ be a $G$-finite model for $\underline E G$ of dimension $n+1$. Further let $H_{i_k}$ be the isotropy group of $i$-simplices of $X$. 
    Then for every $i\in \{0,\cdots,n+1\}$, we have
    $$\mathrm{Hom}_{\mathbb{C}G}(\mathbb{C}[SX^{i}],l^2(G))\cong \bigoplus_{k\in \lambda_i}l^2(G/H_{i_k}).$$
\end{lemma}

\begin{proof}
    Considering $l^2(G)^H=l^2(G/H)$ where $H$ is a subgroup of $G$ acting on $\ell^2(G)$ by right translation, one sees that
    $$
    \mathrm{Hom}_{\mathbb{C}G}(\mathbb{C}[G/H],l^2(G))\cong\mathrm{Hom}_{\mathbb{C}G}(\mathbb{C}G,l^2(G/H)).
    $$
    Following the proof of Proposition \ref{proj resolution}, write
    $$
    \bC[SX^i]= \bigoplus_{k\in \lambda_i} \bC[G/H_{i_k}],
    $$
     where $H_{i_k}$ is the isotropy group of a $i$-simplex in $SX$.
    Putting these together we infer that
    \begin{align*}
    \mathrm{Hom}_{\mathbb{C}G}(\mathbb{C}[SX^{i}],l^2(G))
    =& \mathrm{Hom}_{\mathbb{C}G}(\bigoplus_{k\in \lambda_i} \bC[G/{H_{i_k}}],l^2(G))\\
        \cong& \bigoplus_{k\in \lambda_i}\mathrm{Hom}_{\mathbb{C}G}(\bC[G/{H_{i_k}}],l^2(G))\\
        \cong& \bigoplus_{k\in \lambda_i}\mathrm{Hom}_{\mathbb{C}G}( \bC G,l^2(G/{H_{i_k}}))\\
        \cong& \bigoplus_{k\in \lambda_i}l^2(G/{H_{i_k}}).
    \end{align*}
    This finishes the proof.
\end{proof}

Given a suitable cochain complex of Hilbert spaces, we study two specific direct sum decompositions associated with it. A comparison between the associated projections derived from them leads us  obtaining the  desired description of higher Kazhdan projections.

Our finiteness assumption implies the existence of a finite-length resolution, and consequently, a finite-length cochain complex of Hilbert spaces and bounded linear operators: 
   \begin{equation}
       \label{cochain complex}
       C^0 \stackrel{d_0}{\longrightarrow} C^1 \stackrel{d_1}{\longrightarrow} \cdots \stackrel{d_{n}}{\longrightarrow} C^{n+1}.
   \end{equation}
Focusing on each term $C^i$ for $i = 0, \ldots, n+1$, we obtain a sequence of Hilbert spaces
   \begin{equation*}
   \label{Hseq2}
       0 \stackrel{}{\longrightarrow} \mathrm{ker}d_{i} \longrightarrow  C^i \stackrel{d_i}{\longrightarrow} \overline{\mathrm{im}d_{i}} \stackrel{}{\longrightarrow} 0.
   \end{equation*}
The sequence is, in general, not exact. However, using the polar decomposition for bounded linear operators on Hilbert spaces, one can modify it to obtain an exact sequence. 
\begin{equation}
   \label{Hseq3}
       0 \stackrel{}{\longrightarrow} \mathrm{ker}d_{i} \longrightarrow  C^i \stackrel{\tilde{d_i}}{\longrightarrow} \overline{\mathrm{im}d_{i}} \stackrel{}{\longrightarrow} 0.
   \end{equation}
Indeed, for each differential $d_i : C^i \to \overline{\mathrm{im} d_i}$, the polar decomposition yields a unitary operator $\tilde{d}_i$ on $C^i$ such that
$$
   d_i=\tilde{d_i} |d_i|,
   $$
   satisfying $\mathrm{ker}(\tilde{d_i})=\mathrm{ker}(|d_i|)=\mathrm{ker}d_i$, and $\mathrm{im}\tilde{d_i}=\overline{\mathrm{im}d_i}$.
The exactness of the sequences in \eqref{Hseq3} then follows.

\begin{lemma} (See \cite[Section 3]{Bader-Nowak})
    \label{lemma: kernel of Delta}
    For a cochain complex of Hilbert spaces and bounded linear operators as in \eqref{cochain complex}, one has the following orthogonal direct sum decompositions. 
    \begin{equation} 
    \label{direct sum: kerdi} 
     \ker d_i = \overline{\mathrm{im}\, d_{i-1}} \oplus \ker\Delta_i, \end{equation} 
     for $i = 0, \ldots, n$, and 
     \begin{equation} \label{direct sum: Cn+1} C^{n+1} = \overline{\mathrm{im}\, d_n} \oplus \ker\Delta_{n+1}, \end{equation} 
     where $\Delta_i = d_{i}^* d_{i} + d_{i-1} d_{i-1}^*$ denotes the higher Laplace operator in degree $i$.
\end{lemma}

\begin{proof}  
    Let $i \leq n+1$. Since the codifferential satisfies $d_id_{i-1}=0$, we have 
    $$(\mathrm{ker}d_{i-1}^*)^\perp=\overline{\mathrm{im}d_{i-1}}\subset \mathrm{ker}d_{i},$$
    where $d_{i-1}^*$ is the adjoint operator of $d_{i-1}$.
    Hence
    $$
    \mathrm{ker}d_{i-1}^* + \mathrm{ker}d_{i}=C^i.
    $$
    Decompose $C^i$ into mutually orthogonal subspaces 
    $$C^i=C^i_-\oplus C^i_0\oplus C^i_+,$$
    where
    \begin{equation*}
        \begin{aligned}
            C^i_-&:=(\mathrm{ker}d_{i-1}^*)^\perp \cap \mathrm{ker}d_{i}=(\mathrm{ker}d_{i-1}^*)^\perp=\overline{\mathrm{im}d_{i-1}},\\
            C^i_0&:=\mathrm{ker}d_{i-1}^* \cap \mathrm{ker}d_{i},\\
            C^i_+&:=\mathrm{ker}d_{i-1}^* \cap (\mathrm{ker}d_{i})^\perp=(\mathrm{ker}d_{i})^\perp=\overline{\mathrm{im}d_{i}^*}.
        \end{aligned}
    \end{equation*}
    It is clear that 
    \begin{equation} \label{ker d_i}
    \mathrm{ker}d_{i}=C^i_-\oplus C^i_0,\quad \mathrm{ker}d_{i-1}^*=C^i_+\oplus C^i_0,
    \end{equation}
    and
    \begin{equation*}
    \overline{\mathrm{im}d_{i}}=C^{i+1}_-,\quad \overline{\mathrm{im}d_{i}^*}=C^{i}_+.
    \end{equation*}
    Further we may write the higher Laplace operator $\Delta_i$
    $$
    \Delta_i=\Delta_i^+ +\Delta_i^-,
    $$
    where 
    $$
    \Delta_i^+:=d_{i}^*d_{i}, \quad \Delta_i^-:=d_{i-1}d_{i-1}^*.
    $$
    Thus,
    \begin{align}
        \label{kerDelta}\mathrm{ker}\Delta_i^+=\mathrm{ker}d_{i}=C^i_-\oplus C^i_0,\quad \mathrm{ker}\Delta_i^-=\mathrm{ker}d_{i-1}^*=C^i_+\oplus C^i_0.
    \end{align}
    The higher Laplace operator $\Delta_i^+$ can be decomposed into 
    $$
    \Delta_i^+:C^i \twoheadrightarrow C^i_+ \stackrel{\bar{\Delta}_i^+}{\longrightarrow} C^i_+ \hookrightarrow C^i,
    $$
    and similarly, $\Delta_i^-$ can be decomposed into 
    $$
    \Delta_i^-:C^i \twoheadrightarrow C^i_- \stackrel{\bar{\Delta}_i^-}{\longrightarrow} C^i_- \hookrightarrow C^i,
    $$ 
    where $C^i \twoheadrightarrow C^i_{\pm}$ is the orthogonal projection, $C^i_{\pm} \hookrightarrow C^i$ is the inclusion, and $\bar{\Delta}_i^+$ and $\bar{\Delta}_i^-$ are both injective operators.
    Here the injectivity of $\bar{\Delta}_i^{\pm}$ follows from the equality (\ref{kerDelta}).
    Considering 
    $$
    \Delta_i=\bar{\Delta}_i^-\oplus 0\oplus \bar{\Delta}_i^+:C^i_-\oplus C^i_0\oplus C^i_+ \to C^i_-\oplus C^i_0\oplus C^i_+,
    $$
    we conclude that
    $$
    \mathrm{ker}\Delta_i=C^i_0.
    $$
    From the above discussion, by substituting the terms $C^i_-$ and $C^i_0$ in (\ref{ker d_i}) with $\overline{\mathrm{im}\, d_{i-1}}$ and $\ker\Delta_i$, respectively, the direct sum equality (\ref{direct sum: kerdi}) follows. The case of degree $n+1$ is similar, and that implies the equality (\ref{direct sum: Cn+1}).
\end{proof}

\begin{lemma}
  \label{split}
   Let $0\leq i\leq n+1$. The short exact sequence given in (\ref{Hseq3}) satisfies
   \begin{equation}
    \label{direct sum: Ci}
       C^i\cong \mathrm{ker}d_{i} \oplus \overline{\mathrm{im}d_{i}}.
   \end{equation}
   Moreover, together with Lemma \ref{lemma: kernel of Delta}, we obtain that for a cochain complex of Hilbert spaces and bounded linear operators \eqref{cochain complex}, there is an associated direct sum equality of vector spaces
   $$
   (\bigoplus_{i \ \mathrm{even}}{C^i}) \oplus (\bigoplus_{i\  \mathrm{odd}}{\mathrm{ker}\Delta_i}) \cong (\bigoplus_{i\  \mathrm{odd}}{C^i}) \oplus (\bigoplus_{i\  \mathrm{even}}{\mathrm{ker}\Delta_i}).
   $$  
\end{lemma} 
\begin{proof}
   Let us recall that short exact sequences of Hilbert spaces always split. Let 
   $$
    0 \longrightarrow E \stackrel{\iota}{\longrightarrow} F \stackrel{p}{\longrightarrow} K \longrightarrow 0
   $$ 
   be an exact sequence of Hilbert spaces and bounded linear operators. We need to show there exists a bounded linear operator $j:K\to F$ such that $pj=\mathrm{id}_K$.
   Since $\iota(E)\subseteq F$ is a closed subspace, therefore there exists an orthogonal complement $S$ such that $F=S\oplus \iota(E)$. 
   Exactness of the sequence implies $\mathrm{ker}p\subseteq \iota(E)$, and  the direct sum decomposition implies $S\cap \iota(E)=0$. Hence $p|_{S}:S\to K$ is injective.
   On the other hand, since $p$ is surjective and $\iota(E)\subseteq \mathrm{ker}p$, we infer $p|_{S}$ is also surjective. Therefore, $p|_{S}:S\to K$ is a bijective bounded linear operator and hence there exists an inverse $j=(p|_S)^{-1}:K\to S\subset F$ such that $pj=\mathrm{id}_K$. 
   
  Using the splitting, we deduce the decomposition claimed in the statement.
   To show the second part of the statement, we appeal to the 
    direct sum decompositions in (\ref{direct sum: kerdi}), (\ref{direct sum: Cn+1}) and (\ref{direct sum: Ci}).
   We show it for the case where $n$ is odd. 
    \begin{equation*}
       \begin{aligned}
        \bigoplus_{i=0}^{\frac{n+1}{2}} C^{2i} \oplus \bigoplus_{i=0} ^{\frac{n-1}{2}}\ker\Delta_{2i+1}
        &\cong
         \bigoplus_{i=0}^{\frac{n-1}{2}}\left(C^{2i}\oplus\ker\Delta_{2i+1}\right)\oplus C^{n+1}\\
         &\cong 
         \bigoplus_{i=0}^{\frac{n-1}{2}} \left( \ker d_{2i}\oplus\overline{\text{im} \, d_{2i}} \oplus \ker\Delta_{2i+1}\right)\oplus C^{n+1}\\
         &\cong
         \bigoplus_{i=0}^{\frac{n-1}{2}}\left(\ker d_{2i} \oplus \ker d_{2i+1}\right)\oplus C^{n+1}
         \\
         &\cong
         \bigoplus_{i=0}^{{n}}\ker d_{i} \oplus C^{n+1},      
      \end{aligned}
    \end{equation*}
 and           
   \begin{equation*}
       \begin{aligned}
            \bigoplus_{i=0}^{\frac{n-1}{2}} C^{2i+1} \oplus \bigoplus_{i=0} ^{\frac{n+1}{2}}\ker\Delta_{2i}
             &\cong
            \bigoplus_{i=0}^{\frac{n-1}{2}} \left(\ker\Delta_{2i} \oplus C^{2i+1}\right) \oplus \ker\Delta_{n+1}\\
            &\cong 
            \left(\ker d_0 \oplus \ker d_1 \oplus \overline{\text{im} \, d_1}\right)
            \oplus\bigoplus_{i=1}^{\frac{n-1}{2}}\left(\ker \Delta_{2i}\oplus \ker d_{2i+1} \oplus \overline{\text{im}\, d_{2i+1}}\right)\oplus \ker\Delta_{n+1}\\
            &\cong
            \ker d_0 \oplus \ker d_1 \bigoplus_{i=1}^{\frac{n-1}{2}} \left( \left(\overline{\text{im} \, d_{2i-1}}\oplus \ker\Delta_{2i}\right) \oplus \ker d_{2i+1} \right) \oplus \overline{\text{im} \, d_{n}} \oplus \ker\Delta_{n+1}\\
            &\cong
            \ker d_0 \oplus \ker d_1 \bigoplus_{i=1}^{\frac{n-1}{2}} \left( \ker d_{2i} \oplus \ker d_{2i+1} \right) \oplus \overline{\text{im} \, d_{n}} \oplus \ker\Delta_{n+1}\\
            &\cong
            \bigoplus_{i=0}^{n}\ker d_i\oplus 
            \left (\overline{\text{im} \, d_{n}} \oplus \ker\Delta_{n+1}\right)
            \\
            &\cong
            \bigoplus_{i=0}^{{n}}\ker d_{i} \oplus C^{n+1}.
        \end{aligned}
    \end{equation*}
 The case where $n$ is even is similar.    
\end{proof}

   \begin{proposition}
   \label{lem: step 2}
       Let $G$ be a discrete group and let $X$ be a $G$-finite model for $\underline E G$ of dimension $n+1$.       
       For each degree $j$, assume that $\Delta_j$ has a spectral gap at $0$ when its kernel is non-zero.
       Then we have the following identity in K-theory.
   $$
    \sum_{i=0}^{n+1}\sum_{k\in \lambda_i}(-1)^i[\rho_{H_{i_k}}]=\sum_{j=0}^{n+1}(-1)^j[p_j]\in \mathrm{K_0(C^*_{red}}(G)),
   $$
   where $\rho_{H_{i_k}}$ is the averaging projection associated with $H_{i_k}$, and $p_j$ is the higher Kazhdan projection in degree $j$.
\end{proposition}
\begin{proof}
   Lemma \ref{split} together with Lemma \ref{step1}, applied to the cochain complex
   $$
     ( \mathrm{Hom}_{\mathbb{C}G}(\mathbb{C}[SX^{\bullet}],l^2(G)), d^{\bullet} )
    \cong ( \bigoplus_{k\in \lambda_{\bullet}}l^2(G/H_{i_k}), d^{\bullet} ),
   $$ 
   imply  
   \begin{equation*}
   \label{direct sum}
      \left(\bigoplus_{i \ \mathrm{even}}\bigoplus_{k\in \lambda_i}l^2(G/H_{i_k}) \right)\oplus \left(\bigoplus_{j\  \mathrm{odd}}{\mathrm{ker}(\Delta_j)}\right) \cong \left(\bigoplus_{i\  \mathrm{odd}}\bigoplus_{k\in \lambda_{i}}l^2(G/H_{i_k}) \right)\oplus \left(\bigoplus_{j\  \mathrm{even}}{\mathrm{ker}(\Delta_j)}\right).
   \end{equation*}
   Denote by $S$ the map implementing the above isomorphism.
   As before denote by  $\rho_{H_{i_k}}$ the projection onto $l^2(G/H_{i_k})$ and by  $p_j$ the projection onto $\mathrm{ker}(\Delta_j)$ (i.e. higher Kazhdan projections). 
   Hence we obtain the following direct sum equality of projections:
   \begin{equation}
   \label{direct sum of proj}
       \left(\bigoplus_{i \ \mathrm{even}}\bigoplus_{k\in \lambda_i} \rho_{H_{i_k}}\right)\oplus \left(\bigoplus_{i\  \mathrm{odd}}{p_j}\right) = S^{-1}\left(\left(\bigoplus_{i\  \mathrm{odd}}\bigoplus_{k\in \lambda_{i}}\rho_{H_{i_k}}\right) \oplus \left(\bigoplus_{j\  \mathrm{even}} {p_j}\right)\right)S.
   \end{equation}
   Clearly $\rho_{H_{i_k}}$ lies in $\Csr(G)$ for each $k\in \lambda_i$, $i\in \{0,\cdots, n+1\}$, and by our assumption $p_j$ belongs to $\Csr(G)^{\oplus k_j}$ for some $k_j$ for each $j$.
   
    We claim that $S$ and $S^{-1}$ belong to some matrix algebra of $\Csr(G)$. 
    This together with the equality (\ref{direct sum of proj}) will then imply the desired $K$-theory equality as stated in the statement.
   Let us now prove our claim concerning $S$ and its inverse.    
   As seen in the proof of Lemma \ref{split}, $S$ is obtained from a
   composition of isomorphisms given in (\ref{direct sum: kerdi}), (\ref{direct sum: Cn+1}) and (\ref{direct sum: Ci}).   
    Identifications in (\ref{direct sum: kerdi}) and (\ref{direct sum: Cn+1}) are clear, and the one in (\ref{direct sum: Ci}) is obtained from the decomposition
    $\mathrm{proj}|_{\mathrm{ker}d_j}\oplus \tilde{d_j}(1-\mathrm{proj}|_{\mathrm{ker}d_j})$,
    where $\mathrm{proj}|_{\mathrm{ker}d_j}$ denotes the orthogonal projection onto $\mathrm{ker}d_j$, and 
    $\tilde{d_j}$ is a partial isometry from the exact sequence (\ref{Hseq3}).
    The restriction $\tilde{d_j}|_{{\mathrm{ker}d_j}^{\perp}}$ to the orthogonal complement of $\mathrm{ker}d_j$ coincides with $\tilde{d_j}(1-\mathrm{proj}|_{\mathrm{ker}d_j})$. 
    Hence, to prove the claim, it suffices to show that both $\mathrm{proj}|_{\mathrm{ker}d_j}\oplus \tilde{d_j}(1-\mathrm{proj}|_{\mathrm{ker}d_j})$ and its inverse belong to some matrix algebra over $\mathrm{C^*_{red}}(G)$ for each $j$.    
    We show first that $\mathrm{proj}|_{\mathrm{ker}d_j}$ belongs to 
    $\mathrm M _{k_j} (\Csr(G))$ for some $k_j$.    
    By assumption, $\Delta_j$ has a spectral gap at 0, so does $d_j^*d_j$. 
    Hence, the orthogonal projection onto the kernel of $d_j^*d_j$, which is 
    $\mathrm{proj}|_{\mathrm{ker}d_j}$, belongs to $\mathrm M _{k_j} (\Csr(G))$ for some $k_j$.
    Next we show $\tilde{d_j}$ as well lies in some matrix algebra over  $\Csr(G)$.    
    In this case, since $d_j$ lies in the matrix algebra of $\mathbb CG$, so does $d_j^*d_j$.     
    Suppose $d_j\in \mathrm M_{k_{j+1}\times k_j} (\mathbb CG)$ for some $k_{j+1}$.     
    We apply the continuous functional calculus to the normal element $d_j^*d_j\in \mathrm M _{k_j} (\mathbb CG)$ in the $\Cs$-algebra $\mathrm M_{k_n} (\Csr(G))$.    
    Let $f$ be a function on the spectrum of $d_j^*d_j$ such that 
    $f(t)=t^{-\frac{1}{2}}$ when $t\neq0$ and vanishes when $t=0$.     
    Since $d_j^*d_j$ has a spectral gap at 0, $f$ is a continuous function.    
    Thus $f(d_j^*d_j)$ belongs to $\mathrm M _{k_j} (\Csr(G))$.    
    The fact that $d_i\in \mathrm M _{k_{j+1}\times k_j} (\mathbb CG) \subset \mathrm M _{k_{j+1}\times k_j}(\Csr(G))$, together with $\tilde{d_j}=d_jf(d_j^*d_j)$ imply that 
    $\tilde{d_j}$ also belongs to 
    $\mathrm M _{k_{j+1}\times k_j}(\Csr(G))$.
    Therefore, the isomorphism in (\ref{direct sum: Ci}), given by the operator $\mathrm{proj}|_{\mathrm{ker}d_j}\oplus \tilde{d_j}(1-\mathrm{proj}|_{\mathrm{ker}d_j})$, lies in some matrix algebra over 
    $\Csr(G)$. The same argument can be applied to $S^{-1}$ to show that it belongs to  some matrix algebra over 
    $\Csr(G)$. In K-theory we infer 
    $$
     \sum_{i \,\mathrm{even}}\sum_{k \in \lambda_i}[\rho _{H_{i_k}}]+\sum_{j\, \mathrm{odd}}[p_j] 
     =
     \sum_{i \,\mathrm{odd}}\sum_{k \in \lambda_i}[\rho_{H_{i_k}}]+\sum_{j \,\mathrm{even}}[p_j] \in \mathrm K_0(\Csr(G)).$$ 
     This finishes the proof.
\end{proof}

   We have now all ingredients to prove our main theorem. 
   The main technical step is to identify the K-theory class of the higher Kazhdan projection with an alternating sum of averaging projections corresponding to the isotropy subgroups arising from the action of $G$ on an appropriate simplicial complex. This is done in Proposition \ref{lem: step 2}.
   
\begin{proof} [Proof of Theorem \ref{main thm1}]
   It suffices to prove that the equality \eqref{p1} holds. The equality \eqref{p2} is then a direct consequence of Lemma \ref{lemma: Euler} based on discussion on page 875 of \cite{EM}.
Consider the finite length cochain complex associated with the simplex $SX$ as before
 \begin{align*} 
		0 \longrightarrow \bC[SX^{n+1}] \stackrel{\partial_{n}}{\longrightarrow} \cdots \stackrel{\partial_1}{\longrightarrow} \bC[SX^1] \stackrel{\partial_0}{\longrightarrow} \bC[SX^0] \stackrel{\epsilon}{\longrightarrow} \bC \to 0,
	\end{align*}
   where $\partial_{0},\cdots,\partial_{n}$ denote the boundary map and $\epsilon$ denotes the augmentation map. After applying the functor $\mathrm {Hom}_{\mathbb C G}(\cdot, \ell^2(G))$ to this sequence, and using the identification from Lemma \ref{step1}
   $$
   \mathrm{Hom}_{\mathbb{C}G}(\mathbb{C}[SX^{i}],l^2(G))\cong \bigoplus_{k\in \lambda_i}l^2(G/H_{i_k}),
   $$
   where $\lambda_i$ is a finite index set associated with $\mathbb{C}[SX^{i}]$,
   we obtain
   $$
    0 \ra \bigoplus_{k\in \lambda_1}l^2(G/H_{1_k}) \ra \bigoplus_{k\in \lambda_2}l^2(G/H_{2_k}) \ra \cdots \ra \bigoplus_{k\in \lambda_{n+1}}l^2(G/H_{{n+1}_k}) \ra 0.
   $$
   Denote by $\rho_{H_{i_k}}$ the projection associated with the finite isotropy group $H_{i_k}$. Higher Kazhdan projection $p_n$ is the projection associated with $\ker \Delta_{n}$. Proposition \ref{lem: step 2}
   identifies the alternating sum involving the K-theory class of $\rho_{H_{i_k}}$'s and the one of $p_n$ up to a sign, i.e.
   \begin{equation*}
        \label{step2}
        (-1)^n[p_{n}]=\sum_{i=0}^{n+1}\sum_{k\in \lambda_i}(-1)^i[\rho_{H_{i_k}}] \in \mathrm K_0(\Csr(G)).
    \end{equation*}
 Equation (\ref{alternating sum}) from Proposition \ref{proj resolution} then implies that
 \begin{equation*}
        \sum_{i=0}^{n+1}\sum_{k\in \lambda_i}(-1)^i[\rho_{H_{i_k}}]
        = \sum_{\sigma\in G \backslash SX}{{(-1)}^{|\sigma|}[\rho_{\sigma}]}  
    \end{equation*}      
Therefore 
   $$
(-1)^n[p_n]= \sum_{\sigma\in G \backslash SX}{{(-1)}^{|\sigma|}[\rho_{\sigma}]}.  
$$
This finishes the proof.   
\end{proof}

\begin{example}
There are groups with exactly one non-vanishing higher Kazhdan projection. For instance, virtually free groups have only one non-zero higher Kazhdan projection in degree one (see Lemma \ref{lem:existence}).
Taking the Cartesian product of a finite group with $n$ times copies of the free group $\mathbb F_2$ give rise to groups with having the only non zero higher Kazhdan projection in degree $n$. This example was discussed in \cite{Pooya-Wang}. More examples  come from groups which have cohomology vanishing-below-the-rank. For such groups the Steinberg representation plays a special role. This representation witnesses the non-vanishing of the cohomology in the top degree.
Consider $\Gamma \leq SL(n+1, \mathbb Q_p)$  or in $PGL(n+1, \mathbb Q_p)$ for large prime $p$ and $n>1$. Then all cohomology groups with unitary representation as coefficient vanish for degrees below $n+1$ and cohomology in degree $n+1$ is reduced. (\cite{Garland73}, \cite{dymarajanuszkiewicz2002}).
Therefore, a higher Kazhdan projection $p_{n}$ exists, and it is the only one.
\end{example}

We end this section with the following remark.

\begin{remark}
When a group \( G \) of type $FP_{n+1}$ admits more than one non-zero higher Kazhdan projection, then 
the combinatorial Euler class $[\mathrm{Eul}^{\mathrm{cmb}}_X]$ associated with $X$ is the preimage of the alternating sum of higher Kazhdan projections for $G$ under the assembly map $\mu_G$, i.e., 
    \begin{equation*}
        \mu_0^G([\mathrm{Eul}^{\mathrm{cmb}}_X])
        =
        \sum_{i=1}^{n}(-1)^i[p_i].
    \end{equation*}
Indeed the proof follows the exact same steps as in Theorem \ref{main thm1}. As our current interest is to explicitly describe the K-theory class of higher Kazhdan projections, and to identify their pre-image under the Baum-Connes assembly map we choose to state our main theorem with the assumption that there is only one non-zero higher Kazhdan projection. Moreover, we have a concrete class of groups (non-amenable finitely generated virtually free groups) satisfying the condition of the theorem. While at the moment we are not aware of concerete groups with more than one non-zero higher Kazhdan projections.
\end{remark} 

\section{Delocalised $\ell^2$-Betti numbers}
\label{sec: Examples}
 Higher Kazhdan projections are tightly related to delocalised $\ell^2$-Betti numbers. In this section we provide calculations for hyperbolic groups with a finite dimensional model for $\underline EG$. In the next section we provide more concrete calculations of this invariant for the class of virtually free groups.

In \cite{Pooya-Wang} the first and third author introduced the notion of delocalised $\ell^2$-Betti numbers for hyperbolic groups or more generally for groups for which there is a smooth subalgebra to which delocalised traces extend. 
Here we recall this definition and some related notions.
Let $G$ be a discrete group and $g \in G$. The delocalised trace $\tau_{\langle{g}\rangle}$ on $\ell^1 (G)$ is the bounded linear map with tracial property 
\[
\tau_{\langle{g}\rangle} \colon \ell^1(G) \ra \mathbb C, \qquad \tau_{\langle{g}\rangle}(f) = \sum _{h\in \langle{g}\rangle} f(h),
\]
where $\langle{g}\rangle$ denotes the conjugacy class of $g$.

Smooth subalgebra $\mathcal S$ of a $\Cs$-algebra is a dense Banach subalgebra which is closed under functional calculus. They enjoy an important feature namely they have the same K-theory as the $\Cs$-algebra.  This in particular gives rise to a K-theory pairing by considering 
\[
    \tau_{\langle{g}\rangle} \colon \mathrm K_0(\Csr(G)) \cong \mathrm K_0(\mathcal S) \ra \mathbb C.
\]
Puschnigg in \cite{puschnigg2010} showed that hyperbolic groups have a smooth subalgebra of $\Csr(G)$ to which delocalised traces extend continuously. 

\begin{definition} \cite[Definition 2.11]{Pooya-Wang}
Let $G$ be a discrete group of type $FP_{n+1}$. Assume that there is a smooth subalgebra $\mathcal S \subseteq  \Csr(G)$ to with the delocalised trace $\tau_g$ extends. Assume further that the K-theory class $[p_n]$ lies in $\mathrm K_0(\Csr (G))$. The n-th delocalised $\ell^2$-Betti number of $G$ is 
     \[
        \beta _{n, \langle{g}\rangle}^{(2)} (G) = \tau_{\langle{g}\rangle} ([p_n])
     \]
     \end{definition}

Taking the perspective of operator algebras one may define $\ell^2$-Betti numbers through the K-theory pairing 
\[\beta_n^{(2)} (G) = \tau([p_n]),\] 
where $\tau$ is the von Neumann trace. This trace is defined by picking the coefficient of $\ell^1(G)$- functions in front of the neutral element of the group. Delocalised $\ell^2$-Betti numbers are then in complete analogy with the $\ell^2$-Betti numbers. 

Let us define
$$\mathcal {F}_G :=  \left\langle \left\{ \frac{1}{|F|} \;\middle|\; F \leq G \text{ finite} \right\}\right\rangle \leq \mathrm Q$$ the additive subgroup of $\mathbb Q$ generated by the inverses of the orders of finite subgroups of $G$.

Our definition of delocalised $\ell^2$-Betti numbers relies on the group being hyperbolic, as discussed above. Therefore, we retain this assumption in the calculations to come. Furthermore, because hyperbolic groups are known to satisfy the Baum-Connes conjecture \cite{lafforgue2009}, the assumptions in the next result reflect this context and differ slightly from those in the previous section.

\begin{corollary} \label{cor: nonzero del betti general}
    Let $G$ be a discrete hyperbolic group with only one higher Kazhdan projection $p_n$. Assume there is a $G$-finite model for $\underline E G$ of dimension $n+1$. Let $g\in G$.
    Then the  n-th delocalised $\ell^2$-Betti number of $G$ satisfies
    \[
       \beta^{(2)}_{n, \langle{g}\rangle} (G) \in 
        \mathcal {F}_G \subseteq \mathbb{Q}.
       \]   
\end{corollary}

\begin{proof}
    Due to Proposition \ref{lem: step 2} we have 
    \begin{equation*}
        (-1)^n[p_n]=\sum_{i=0}^{n+1}\sum_{k\in \lambda_i}(-1)^i[\rho_{H_{i_k}}] \in \mathrm{K_0(C^*_{red}}(G)),
    \end{equation*}
 where $\lambda_i$ denotes the number of $i$-simplices of a suitable simplicial model of $\underline EG$, and $\rho_{H_{i_k}}$ is the averaging projection associated to the finite isotropy group $H_{i_k}$. Now pairing with the delocalised traces $\tau_{g}$ for $g\in G$ yields the rational numbers appearing as an alternating sum of the inverse of the orders of specific finite subgroups. When $g\in G$ has an infinite order, all delocalised $\ell^2$-Betti numbers vanish.
\end{proof}

In the next section, we specify conditions under which delocalised $\ell^2$-Betti numbers are non-vanishing.  See Corollary \ref{cor: nonzero del betti}.
\section{Application: Virtually free groups}
\label{sec: Applications}
In this section, we apply Theorem \ref{main thm1} to the class of non-amenable finitely generated virtually free groups. These groups have exactly one non-zero higher Kazhdan projection. As we will see the standard tree on which they act provides a desired $G$-simplicial complex.
Using this setting we will be able to provide a concrete description for the K-theory class of their higher Kazhdan projection. 

We start by recalling some facts about virtually free groups.




\begin{lemma} (see \cite[Section 5]{Serre} and \cite[Theorem 1]{KarrasPietrowskiSolitar1973})
\label{lemma: f.g. virtually free}
    Let $G$ be a discrete group. The following are equivalent:
    \begin{enumerate}
    \item $G$ is a finitely generated virtually free group,
         \item $G$ is a virtually free group and there exists a tree $X$ with a proper and cocompact action of $G$,
     \item $G$ is an HNN extension of the form 
    \begin{equation}
     \label{equa: HNN}
         \langle t_1,\cdots,t_l,K|\ \text{relations in}\ K, t_1A_1{t_1}^{-1}=B_1,\cdots, t_lA_l{t_l}^{-1}=B_l \rangle,
     \end{equation}
     where $K$ is a tree product of a finite number of finite groups (the vertices of $K$), and each $A_i$, $B_i$ is a subgroup of a vertex of $K$. Here the choices of $K$ and subgroups $A_i$ and $B_i$ are unique up to group conjugacy.
    \end{enumerate}
\end{lemma}

\begin{lemma} \cite[Lemma 4.1]{Pooya-Wang} \label{lem:existence}
		Let $G$ be a non-amenable virtually free group. Then $p_1$ lies in matrices over $\Csr(G)$, its K-class belongs to
		$\mathrm{K_0}(\Csr(G))$ and it is non-zero. All other $p_n$'s vanish when $n\neq 1$.
	\end{lemma}
	
	\begin{proof}
		Assume that some free group $\mathbb F_k\ (k>1)$ is a finite index subgroup of $G$. To establish the presence of a spectral gap for $\Delta_1$, it suffices to show the $\ell^2$-cohomology in degree 1 and 2 is reduced \cite[Proposition 16]{Bader-Nowak}. The non-amenability of $G$ implies that $\mathrm H^1(G, \ell^2 (G))$ is reduced, and  
		Shapiro lemma implies that $\mathrm H^2(G, \ell^2(G))$ is reduced (even vanishes) as it is so for $\mathbb F_k$. Consequently, $p_1$ belongs to matrices over $\Csr(G)$. 
		Recall that
		$\tau([p_1]) = \beta^{(2)}_{1}(G)$. Since $\beta^{(2)}_{1}(G)$ is non-zero, it follows that $[p_1]$ is non-zero. Additionally, given the faithfulness of $\tau$, we note that $\beta^{(2)}_{n}(G)$ vanishes for $n> 1$, leading to the vanishing of $p_n$.
	\end{proof}

Lemma \ref{lemma: f.g. virtually free} and Lemma \ref{lem:existence} together with Theorem~\ref{main thm1} give rise to the following corollary.

\begin{corollary}
\label{cor: tree}
    Let $G$ be a non-amenable finitely generated virtually free group, and
    let $X$ be its Bass-Serre tree. Then the $\mathrm{K}$-theory class $[p_1]$ can be described in terms of the projections associated to the vertex and edge groups of the fundamental domain of the tree $X$
    \begin{equation*}
        -[p_1]=\sum_{v\in \, \mathrm{vert}\ G \backslash X}{[\rho_{v}]}\,\,\,\,\,-\sum_{e\in \, \mathrm{edge}\  G \backslash X}{[\rho_{e}]}\in \mathrm{K_0(C^*_{red}}(G)).
    \end{equation*}
\end{corollary}

\begin{remark}
    By Lemma \ref{lemma: f.g. virtually free} a finitely generated virtually free group $G$ can be represented as an HNN extension. Taking this perspective, the K-class of $[p_1]$ paired with the canonical trace yields the generalised Schreier rank formula (see Theorem 2 in \cite{KarrasPietrowskiSolitar1973}). We explain it here.
    Let $\mathrm{vert}\ K$ and $\mathrm{edge}\ K$ denote the set of vertices and edges of the tree $K$, respectively.
    For each $v\in \mathrm{vert}\ K$ and $e\in \mathrm{edge}\ K$, let $K_v$ and $K_e$ be the associated vertex and edge group, respectively. Let $|A_i|=a_i$, $|K_{e_i}| = e_i$ and $|K_{v_i}|=v_i$, with $s+1$ being the number of vertices in $K$. 
    Then we have
    \begin{equation}
    \label{equa: p1 for HNN}
       [p_1]=-\sum_{v\in \mathrm{vert}\ G \backslash X}[\rho_{K_v}]+\sum_{e\in \mathrm{edge}\ G \backslash X}[\rho_{K_e}]+\sum_{i=1}^{l}[\rho_{A_i}].
    \end{equation}
    Indeed, by Lemma \ref{lemma: f.g. virtually free}, the isotropy groups of vertices of the Bass-Serre tree $X$ are the vertex groups $K_v$, while the isotropy groups of edges of $X$ consist of the edge groups $K_e$ together with the subgroups $\{A_i\}_{i\in \{1,\cdots,l\}}$.
    Therefore, the equality (\ref{equa: p1 for HNN}) follows immediately from Corollary \ref{cor: tree}.           
    Applying the canonical trace $\tau$ to both sides of (\ref{equa: p1 for HNN}), we obtain the first $\ell^2$-Betti number of $G$
    $$
    \beta^{(2)}_1(G)=-\frac{1}{v_1}-\cdots-\frac{1}{v_{s+1}}+\frac{1}{e_1}+\cdots+\frac{1}{e_s}+\frac{1}{a_1}+\cdots+\frac{1}{a_l}.
    $$
    The Euler characteristic of $G$ is $\chi(G)=(1-r)/j$, which is independent of the choice of free subgroup $H$
    \cite{Wall}. By Euler-Poincar\'e formula, we have
    $$
    \chi(G)=(1-r)/j=\sum_{i\geq0}(-1)^i\beta^{(2)}_i(G)=-\beta^{(2)}_1(G),
    $$
    implying
    $$  (1-r)/j = -\beta^{(2)}_1(G),$$
    which recovers the generalised Schreier rank formula (see \cite{KarrasPietrowskiSolitar1973})
    \begin{equation}
    \label{equa: KPS}
        r=j(-\frac{1}{v_1}-\cdots-\frac{1}{v_{s+1}}+\frac{1}{e_1}+\cdots+\frac{1}{e_s}+\frac{1}{a_1}+\cdots+\frac{1}{a_l})+1.
    \end{equation}
    \end{remark}

Virtually free groups can be constructed iteratively via amalgamated free products and HNN extensions over finite groups. We next spell out a description for the K-theory class $[p_1]$ associated to the building block of virtually free groups. 

\begin{corollary} \label{cor: amalgam}
    Let $G=F\ast_K L$ be an amalgamated free product of finite groups $F$ and $L$ over a common proper subgroup $K$ (where $K\subsetneqq F$ and $K\subsetneqq L$), with indices $[F:K]\geq 3$ or $[L:K]\geq 3$.
    Then the $K$-theory class $[p_1] \in \mathrm K_0(\Csr(G))$ is
    $$
    [p_1]=
    -[\rho_F]-[\rho_L]+[\rho_K],
    $$
    where $\rho_F$, $\rho_L$, and $\rho_K$ are the averaging projections associated to $H, L$ and $K$, respectively.
\end{corollary}
\begin{proof}
 We first show that the properness of $K$ and the index condition imply the non-amenability of $G$.
 Indeed, the fundamental domain in this case is a segment consisting of an edge and two vertices with stabilizer groups given by $K, F$, and $L$, respectively.
    Hence the Euler characteristic of $G$ in this case is
    $$
    \chi(G)=\frac{1}{|F|}+\frac{1}{|L|}-\frac{1}{|K|},
    $$
    (see \cite{Serre} for example).
    It is strictly negative under the assumptions, while the Euler characteristic and all the $\ell^2$-Betti numbers vanish for any infinite amenable group (see \cite[Theorem 0.2]{Cheeger-Gromov}).
    Therefore, by Corollary \ref{cor: tree}, we obtain the $K$-theoretical formula in the conclusion.
\end{proof}

    
\begin{corollary}
\label{cor: HNN}
    Let $F$ be a finite group with two proper isomorphic subgroups $A,B\subsetneqq F$. 
    Let $G$ be the associated HNN extension. 
    Then the $K$-theory class $[p_1] \in \mathrm K_0(\Csr(G))$ is
    $$
    [p_1]=
    -[\rho_F]+[\rho_A]=-[\rho_F]+[\rho_B].
    $$
\end{corollary}
\begin{proof}
    We first show that the properness of $A$ (and $B$) in $F$ implies the non-amenability of $G$.
    In this case, the fundamental domain is a loop consisting of an edge and a vertex, with stabilizer groups given by $A$ or $B$, and $F$, respectively.
    Hence the Euler characteristic of $G$ is
    $$
    \chi(G)=\frac{1}{|F|}-\frac{1}{|A|}\neq 0,
    $$
    (see \cite{Serre} for example),  while the Euler characteristic and all the $\ell^2$-Betti numbers vanish for any infinite amenable group (see \cite[Theorem 0.2]{Cheeger-Gromov}).
    Therefore, by Corollary \ref{cor: tree}, we obtain the $K$-theoretical formula in the conclusion.
\end{proof}

We end this section with the following concrete example, which verifies the result obtained by direct computation (see \cite{Ren}).

\begin{example} \label{SL2Z+p1}
    Let $G = \mathrm{SL}(2, \mathbb Z) = 
    \mathbb Z_4 *_ {\mathbb Z_2} \mathbb Z_6$, with $u,s,$ and $t$ being the generators of $\mathbb Z_2, \mathbb Z_4,$ and $\mathbb Z_6$, respectively.  
    Corollary \ref{cor: amalgam} then implies that
    \[ [p_1] = \left[\frac{1+u}{2}\right]- \left[\frac{1+s+s^2+s^3}{4}\right] - \left[\frac{1+t+t^2+t^3+ t^4 +t^5}{6}\right].
    \]
\end{example}

Virtually free groups have vanishing reduced $\ell^2$-cohomology in degrees greater than one, which implies that their delocalised $\ell^2$-Betti numbers also vanish in those degrees. In addition, infinite groups have vanishing zeroth $\ell^2$-Betti number, and hence their delocalised zero-th $\ell^2$-Betti numbers vanish as well. Therefore, any non-vanishing of delocalised $\ell^2$-Betti numbers for finitely generated virtually free groups must occur in degree one. Theorem~\ref{main thm1} allows us to establish such non-vanishing results.

\begin{corollary} \label{cor: nonzero del betti}
    Let $G$ be a non-amenable finitely generated virtually free group acting properly on a tree $X$, and let 
    $g\in F$ be an element in a finite subgroup $F \leq G $ that fixes vertices, but no edges of $X$.
    Then the first delocalised $\ell^2$-Betti numbers of $G$ are non-zero and they satisfy
    \[
       \beta^{(2)}_{1, \langle{g}\rangle} (G) \in 
        \mathcal {F}_G \subseteq \mathbb{Q}.
       \]
\end{corollary}

\begin{proof}
Let $g\in G$ be an element that fixes at least one vertex of the tree $X$, but does not fix any edge. Then $g$ is conjugate to an element of the stabilizer of a vertex $v\in SX$, where $SX$ denotes the set of simplices in the quotient tree $X \backslash G$.
By Corollary \ref{cor: tree}, we have
 $
        -[p_1]=\sum_{v\in \,\mathrm{vert}\ G \backslash X}{[\rho_{v}]}\,\,\,\,\,-\sum_{e\in \, \mathrm{edge}\  G \backslash X}{[\rho_{e}]},
 $  
where $\rho_ v$ and $\rho_e$ are averaging projections.
Applying the delocalised trace $\tau_{\langle{g}\rangle}$ to both sides yields
\begin{equation*}
    \beta^{(2)}_{1, \langle{g}\rangle} (G)
    = \tau_{<g>}([p_1]) 
    = - \sum_{v \in \,\mathrm{vert}\ G \backslash X} 
    \tau_{<g>}([\rho_ v]).
    \end{equation*}
Since $g$ does not appear in the isotropy group of any edge, the above sum is taken only over the vertices in $G\backslash SX$. Therefore, no cancellations can occur due to contributions of having different dimensions and thus different signs.
Furthermore, $\beta^{(2)}_{1, \langle{g}\rangle} (G)$ is non-zero since $g$ does fix some vertices, and it clearly belongs to the mentioned subgroup of rationals.
\end{proof}

Clearly when $g\in G$ does not fix any simplices, then the delocalised $\ell^2$-Betti number associated to that $g$ vanishes.  
Let us assume that $g\in G$ fixes some edge and not just vertices. In this case, vanishing and non-vanishing of delocalised $\ell^2$-Betti numbers can occur. Consider an HNN extension $G=F_{*_{\psi}}$ associated with a finite group $F$ with two isomorphic subgroups $A,B\subset F$ and an isomorphism $\psi:A\to B$. In this case, the quotient $G\backslash X$ of the Bass-Serre tree $X$ associated with $G$ is a loop. The isotropy group of the unique vertex and the edge are $F$ and $A$, respectively.
Consider an element $g\in A\subset G$ that fixes the edge.
On the one hand, let $A=B=F$ with $\psi$ being the identity map. Corollary \ref{cor: HNN} implies that the class $[p_1]$ for $G$ vanishes and so does $\beta^{(2)}_{1, \langle{g}\rangle} (G)$.
On the other hand, let $A=B$ be a proper subgroup of a finite abelian group $F$ with $\psi$ being the identity map, then $\beta^{(2)}_{1, \langle{g}\rangle} (G)=1/{|A|}-1/{|F|}$ is strictly positive by Corollary \ref{cor: HNN}. 

We close this article by calculating some examples of infinite groups with with non-vanishing delocalised $l^2$-Betti numbers. 
\begin{example}
Consider the group $G = \mathrm{SL}(2, \mathbb Z) = 
    \mathbb Z_4 *_ {\mathbb Z_2} \mathbb Z_6$, with $u,s,$ and $t$, generators for $\mathbb Z_2, \mathbb Z_4,$ and $\mathbb Z_6$, respectively. From Example \ref{SL2Z+p1} we have 
$$ [p_1] = \left[\frac{1+u}{2}\right]- \left[\frac{1+s+s^2+s^3}{4}\right] - \left[\frac{1+t+t^2+t^3+ t^4 +t^5}{6}\right].
$$
The delocalised $\ell^2$-Betti numbers for $G$ are then
		\begin{equation*}
			\beta^{(2)}_{1, \langle{g}\rangle}(G) = 
			\begin{cases}
				1/12 &\qquad  g=e\\
                1/12  &\qquad  g\in\langle{u}\rangle= \langle{s^2}\rangle = \langle{t^3}\rangle\\
				-1/4  &\qquad  g\in\langle{s}\rangle\\
				-1/6  &\qquad  g\in\langle{t}\rangle\\     
                -1/4  &\qquad  g\in\langle{s^3}\rangle\\   
				-1/6  &\qquad  g\in\langle{t^2}\rangle\\
                -1/6  &\qquad  g\in\langle{t^4}\rangle\\
                -1/6  &\qquad  g\in\langle{t^5}\rangle\\
				0    &\qquad \text{otherwise} 
			\end{cases}
		\end{equation*}
		and $\beta ^{(2)}_{k, \langle{g}\rangle}(G) = 0$ for $k \neq 1$ and $g\in G$.
\end{example}

\begin{example}
    Consider the Klein four group $V=\langle a,b\,|\,a^2=b^2=e,ab=ba \rangle$
     with its two isomorphic proper subgroups $A=\{e,a\}$ and $B=\{e,b\}$. 
    Let $G$ be
     the associated HNN extension.
    Then by Corollary~\ref{cor: HNN}, $[p_1]$ is 
    $$
    [p_1]=\left[\frac{e+a}{2}\right]-\left[\frac{e+a+b+ab}{4}\right].
    $$
    Hence, the delocalised $\ell^2$-Betti numbers are
    \begin{equation*}
        \beta^{(2)}_{1, \langle{g}\rangle}(G) = 
        \begin{cases}
            1/4 &\qquad g=e\\
            -1/4 &\qquad g=ab \\
            0 &\qquad \text{otherwise}
        \end{cases}
    \end{equation*}
    and $\beta ^{(2)}_{k, \langle{g}\rangle}(G) = 0$ for $k \neq 1$ and $g\in G $.
\end{example}

\begin{example}
    Consider the dihedral group $D_4=\langle r,s \,|\,r^4=s^2=e,srs=r^{-1}\rangle
    $
    with its two isomorphic proper subgroups $A=\{e,s \}$ and $B=\{e,sr \}$.
    Let $G
    $ be the associated HNN extension.
   Then by Corollary~\ref{cor: HNN}, the $K$-theory class $[p_1]$ 
   is
    $$
      [p_1]=\left[\frac{e+s}{2}\right]-\left[\frac{e+r+r^2+r^3+s+sr+sr^2+sr^{3}}{8}\right].
    $$
    Hence, the delocalised $\ell^2$-Betti numbers are
    \begin{equation*}
        \beta^{(2)}_{1, \langle{g}\rangle}(D_4^*)=
        \begin{cases}
            3/8 &\qquad g=e \\
            -1/4 &\qquad g\in \langle {r} \rangle=\langle {r^3} \rangle \\
            -1/8 &\qquad g\in \langle {r^2} \rangle \\
            1/4 & \qquad g\in \langle {s}\rangle \\
            -1/4 & \qquad g\in \langle {sr} \rangle \\
            0 &\qquad \text{otherwise}
        \end{cases}
    \end{equation*}
    and $\beta ^{(2)}_{k, \langle{g}\rangle}(G) = 0$ for $k \neq 1$ and $g\in G$.
\end{example}
	
	
	\bibliographystyle{alpha}
	\bibliography{main}

	\vspace{2em}
	\begin{minipage}[t]{0.45\linewidth}
		\small
		Sanaz Pooya \\
		Institute of Mathematics\\
		University of Potsdam\\
		14476 Potsdam, Germany\\
		{\footnotesize sanaz.pooya@uni-potsdam.de}
		\\   
        
        \small 
		Baiying Ren \\ Research Center of Operator Algebras \\
		East China Normal University\\
		Shanghai 200241, China \\
		{\footnotesize 52275500020@stu.ecnu.edu.cn}
		\\ 
        
		\small 
		Hang Wang \\ Research Center of Operator Algebras \\
		East China Normal University\\
		Shanghai 200241, China \\
		{\footnotesize wanghang@math.ecnu.edu.cn}
	\end{minipage}
\end{document}